\documentclass{amsart}

\usepackage{amsmath,amsthm}
\usepackage{enumerate,centernot}

\theoremstyle{plain}

\newtheorem{theorem}{Theorem}          
\newtheorem{lemma}[theorem]{Lemma}         
\newtheorem*{lemma*}{Lemma}         
\newtheorem{propn}[theorem]{Proposition}         
\newtheorem{cor}[theorem]{Corollary} 

\theoremstyle{remark}                      
\newtheorem{remark}[theorem]{Remark}
\newtheorem{example}[theorem]{Example}

\newcommand\sm{\setminus}
\newcommand\col{\colon}
\newcommand\sub{\subseteq}

\newcommand{\set}[2]{\{\,{\textstyle#1}:\,{\textstyle #2}\,\}}

\newcommand\conj{\overline}

\newcommand\lone{L^1}
\newcommand\ltwo{L^2}
\newcommand\linfty{L^\infty}
\newcommand{\vn}{\mathit{VN}}

\newcommand{\A}{A}              
\newcommand\B{\mathcal{B}}

\newcommand\freegrp{\mathbb{F}}
\newcommand\dom{D}

\newcommand{\G}{\mathbb{G}}

\newcommand\N{\mathcal{N}}
\newcommand\nat{\mathbb{N}}
\newcommand{\complex}{\mathbb{C}}          

\newcommand{\cop}{\Delta}
\newcommand{\id}{\mathrm{id}}

\newcommand\ot{\otimes}
\newcommand\vnot{\mathop{\overline\otimes}}

\newcommand\dual[1]{\hat{#1}}

\newcommand\conv{\star}

\newcommand\wW{\mathcal{W}}
\newcommand\ww{W}
\newcommand{\mc}{\mathcal}

\newcommand{\ip}[2]{\langle{#1},{#2}\rangle}
\newcommand{\Bigip}[2]{\Bigl\langle{#1},{#2}\Bigr\rangle}
\newcommand\pp[2]{(#1\mid #2)}
\newcommand\bigpp[2]{\bigl(#1\bigm| #2\bigr)}
\newcommand\Bigpp[2]{\Bigl(#1\Bigm| #2\Bigr)}

\begin{document}

\title[Completely positive definite functions and Bochner's theorem]%
      {Completely positive definite functions and Bochner's theorem 
       for locally compact quantum groups}
\author{Matthew Daws and Pekka Salmi}

\address{School of Mathematics, University of Leeds,
Leeds LS2 9JT, United Kingdom}

\email{matt.daws@cantab.net}

\address{Department of Mathematical Sciences, University of Oulu,
PL~3000, FI-90014 Oulun yliopisto, Finland}

\email{pekka.salmi@iki.fi}

\begin{abstract}
We prove two versions of Bochner's theorem for locally compact
quantum groups.
First, every completely positive definite ``function'' 
on a locally compact quantum group $\G$ arises as a transform of 
a positive functional on the universal C*-algebra $C_0^u(\dual\G)$
of the dual quantum group.
Second, when $\G$ is coamenable, complete positive definiteness may be 
replaced with the weaker notion of positive definiteness, which 
models the classical notion. A counterexample is given to show 
that the latter result is not true in general. 
To prove these results, we show two auxiliary results of independent
interest: products are linearly dense in $\lone_\sharp(\G)$,
and when $\G$ is coamenable, the Banach $*$-algebra 
$\lone_\sharp(\G)$ has a contractive bounded approximate identity.


\end{abstract}

\maketitle

\section{Introduction}

Bochner's theorem (as generalised by Weil) tells us that any positive definite
function on a locally compact abelian group $G$ is the Fourier--Stieltjes
transform of a positive measure on the dual group $\dual G$.  In non-abelian
harmonic analysis, we can replace the algebra $C_0(\dual G)$ by the group
C$^*$-algebra $C^*(G)$, and hence replace positive measures on $\dual G$ by
positive functionals on $C^*(G)$.  Viewing $C^*(G)^*$ as $B(G)$, the
Fourier--Stieltjes algebra, Bochner's theorem essentially
says that positive definitive functions are precisely the positive elements of
$B(G)$ (this viewpoint is taken in \cite[D\'efinition~2.2]{eymard}).

For a locally compact quantum group $\G$, we replace functions on groups by
elements of von Neumann (or C$^*$-) algebras, which come equipped with extra
structure reminiscent of an algebra which really arises from a group.  Given
$\G$, we can form the universal dual algebra $C_0^u(\dual\G)$, which generalises
the passage from $G$ to the full group C$^*$-algebra $C^*(G)$ (see the next
section for further details on $C_0^u(\dual\G)$ and so forth).  Letting
$\wW\in M(C_0(\G)\otimes C_0^u(\dual\G))$ be the maximal unitary
corepresentation of $\G$, we have an algebra homomorphism $C_0^u(\dual\G)^*
\rightarrow M(C_0(\G))\subseteq L^\infty(\G); \dual\mu\mapsto
(\id\otimes\dual\mu)(\wW^*)$.
In the commutative case, the image is precisely the Fourier--Stieltjes algebra
(we remark that it is slightly a matter of convention if one uses $\wW$ or
$\wW^*$ here).  Motivated by this, there are perhaps two obvious notions
for what a ``positive definite'' element of $L^\infty(\G)$ should be; here
we introduce some of our own terminology:
\begin{enumerate}[(1)]
\item\label{defn:pdf} 
  A \emph{positive definite function} is
  $x\in \linfty(\G)$ with $\ip{x^*}{\omega \conv \omega^\sharp}\geq 0$
  for all $\omega\in \lone_\sharp(\G)$.
\item\label{defn:fsp} 
  A \emph{Fourier--Stieltjes transform of a positive measure}
  is $x\in \linfty(\G)$ such that there exists $\hat\mu\in C_0^u(\dual\G)^*_+$
  with $x = (\id\otimes\hat\mu)(\wW^*)$.
\end{enumerate}
It seems that definition (\ref{defn:fsp}) is a better fit with 
the current literature, although the term  ``positive definite
function'' is not commonly used in this context
(for examples where positive functionals in
$\hat\mu\in C_0^u(\dual\G)^*_+$, or their transforms, are 
used in place of positive definite functions (from the classical case),
see \cite{kalantar-neufang-ruan} which studies Markov operators,
\cite[Section~4]{brannan:quan_auts} and \cite{brannan:approx}
which study various approximation properties for von Neumann algebras over
quantum groups, or \cite{kyed:prop-T} and \cite{kyed-soltan} 
which study property~$(T)$ for quantum groups; the latter reference
actually uses the term ``positive definite function'' in an offhand way).
Definition (\ref{defn:pdf}) is the most natural as it directly generalises
the notion of a positive definite function on $L^1(G)$, see for example
\cite[Theorem~13.4.5]{dixmier:C*-algebras}.
This definition is rather briefly studied for Kac algebras in 
\cite[Section~1.3]{enock-schwartz:kac}; however, it is mainly
(\ref{defn:fsp}), in various guises, which is 
used in \cite{enock-schwartz:kac}.  Indeed, we show in
Example~\ref{ex:not-cp} below that even in the cocommutative case,
definition (\ref{defn:pdf}) is problematic without some sort of
amenability assumption -- to be precise, that $\G$ is coamenable.
Even when $\G$ is coamenable, we are required to deal with the unbounded
antipode $S$, and our techniques are necessarily different from those
used for Kac algebras.  We remark that it is easy to see that always
(\ref{defn:fsp})$\implies$(\ref{defn:pdf}), see
Lemma~\ref{lem:fsp_implies_pdf} below.

Definition (\ref{defn:fsp}) applied in the cocommutative case suggests that
a ``positive definite'' element of $\mathit{VN}(G)$ should come from a
positive measure in $M(G) = C_0(G)^* = \mathit{ML}^1(G)$, the
multiplier algebra of $L^1(G)$. On the dual side, De Canni\`ere and
Haagerup showed in  
\cite{de-canniere-haagerup:multipliers} that \emph{completely
positive} multipliers of $A(G)$ coincide with positive definite functions on
$G$.  This suggests the following notions:
\begin{enumerate}[(1)]\setcounter{enumi}{2}
\item\label{defn:cpm} 
  A \emph{completely positive multiplier}
  is $x\in \linfty(\G)$ such that there exists 
  a completely positive left multiplier $L_x\col
  \lone(\dual\G)\to\lone(\dual\G)$ 
  with $x \hat\lambda(\hat\omega) = \hat\lambda(L_x(\hat\omega))$
  for every $\hat\omega\in \lone(\dual\G)$. 
  Here $\hat\lambda$ denotes 
  the map $\hat\omega\mapsto (\id\ot\hat\omega)(W^*)
  = (\hat\omega\ot\id)(\hat W)$
  where $W\in M(C_0(\G)\ot C_0(\dual\G))$ is 
  the left multiplicative unitary of $\G$. 
\item\label{defn:cpdf} 
  A \emph{completely positive definite function}
  is $x\in \linfty(\G)$ such that there exists a normal 
  completely positive map 
  $\Phi\col\B(\ltwo(\G)) \to \B(\ltwo(\G))$ with
  \[ 
  \ip{x^*}{\omega_{\xi, \alpha} \conv \omega_{\eta, \beta}^{\sharp}}
  = \pp{ \Phi(\theta_{\xi,\eta})\beta}{ \alpha }
  \]
  for every $\xi,\eta\in\dom(P^{1/2})$ and 
  $\alpha,\beta\in\dom(P^{-1/2})$. Here $P$ is a
  densely defined, positive, injective operator on $\ltwo(\G)$
  implementing the scaling group of $\G$. 
\end{enumerate}
The first named author showed in \cite{daws:cp-multipliers} that
(\ref{defn:cpm}) and (\ref{defn:fsp}) are equivalent notions, and that they
imply (\ref{defn:cpdf}).  We note that while (\ref{defn:fsp}) is the notion
mostly adopted in the literature (see above), it is actually the map $L_x$
(or its adjoint) which is of interest (the point being that the implication
(\ref{defn:fsp})$\implies$(\ref{defn:cpm}) is very easy to establish).
Let us motivate (\ref{defn:cpdf}) a little more.
The unbounded involution~$\sharp$ on $L^1(\G)$ is given by $\omega^\sharp
= \omega^* \circ S$, where this is bounded.  Normal completely positive maps
on $\mc B(L^2(\G))$ biject with the positive part of the extended (or
weak$^*$) Haagerup tensor product 
$\mc B(L^2(\G)) \overset{eh}{\otimes} \mc B(L^2(\G))$
(see \cite{blecher-smith,effros-ruan}), where $a\otimes b$ is
associated to $\Phi$ with 
\[ \Phi(\theta) = a\theta b \quad\Leftrightarrow\quad
\pp{ \Phi(\theta_{\xi,\eta})\beta}{ \alpha } = \ip{a}{\omega_{\xi,\alpha}}
\ip{b}{\omega_{\beta,\eta}}. \]
Hence (\ref{defn:cpdf}) is equivalent to the existence of a positive
$u\in \mc B(L^2(\G)) \overset{eh}{\otimes} \mc B(L^2(\G))$ with
\[ \ip{\Delta(x^*)}{\omega_1 \otimes \omega_2^*\circ S}
= \ip{u}{\omega_1\otimes\omega_2^*}. \]
Hence, informally, this is equivalent to $(\id\otimes S)\Delta(x^*)$ being
a positive member of $\mc B(L^2(\G)) \overset{eh}{\otimes}\mc B(L^2(\G))$.
In the commutative case, $x^*=F\in L^\infty(G)$ say, and this says that the
function $(s,t) \mapsto F(st^{-1})$ is, in some sense, a ``positive kernel'',
i.e. that $F$ is positive definite.
So (\ref{defn:cpdf}) says that $x^*$ defines a non-commutative,
positive definite kernel.

We remark that as the ``inverse'' operator on $\G$
(the antipode $S$) and the adjoint~$*$ on $L^\infty(\G)$ do not commute,
we have to be a little careful about using $x$ or $x^*$ in the above
definitions.

The principle results of this paper are:
\begin{itemize}
\item For any $\G$, we have that (\ref{defn:cpdf}) is equivalent to
(\ref{defn:cpm}) and hence equivalent to (\ref{defn:fsp}).
\item When $\G$ is coamenable (as is true in the commutative case!) all
four conditions are equivalent.
\end{itemize}
Both these results may be interpreted as versions of Bochner's 
theorem for locally compact quantum groups.
When given a condition like (\ref{defn:pdf}) the obvious thing to try is
a GNS construction, in this case applied to the $*$-algebra $L^1_\sharp(\G)$.
In general, this algebra does not have an approximate identity, so
we first show in Section~\ref{sec:density} that products are always linearly
dense in $L^1_\sharp(\G)$, which enables a suitable GNS construction.
In Section~\ref{sec:cpd} we apply this, together with techniques similar to
those used in \cite{daws:cp-multipliers} to show that 
(\ref{defn:cpdf})$\implies$(\ref{defn:cpm}).  One can take as a definition
that $\G$ is coamenable if and only if $L^1(\G)$ has a bounded approximate
identity.  In Section~\ref{sec:bai_coamen_case} we show that in this case,
also $L^1_\sharp(\G)$ has a (contractive) approximate identity.  Indeed,
we prove a slightly more general statement, adapting ideas of J.~Kustermans,
A.~Van~Daele and J.~Verding from \cite{kustermans:one-parameter} (we wish to
approximate the counit $\epsilon$, which is \emph{invariant} for the scaling
group, and it is this invariance which is key; the argument in
\cite{kustermans:one-parameter} is for multiplier algebras of C$^*$-algebras
and modular automorphism groups, and to our mind, works because the unit of
$M(A)$ is invariant for the modular automorphism group).  We suspect that the
ideas of Sections~\ref{sec:density} and~\ref{sec:bai_coamen_case} will prove
to be useful in other contexts.
In Section~\ref{sec:bochner} we apply this to condition (\ref{defn:pdf}).
In the final section we consider $n$-positive multipliers.

\section{Preliminaries}

Throughout the paper $\G$ denotes a locally compact quantum group
\cite{kustermans:LCQG-chapter, kustermans-vaes:LCQG_VN, kustermans-vaes:LCQG}.
Its comultiplication $\cop$ is implemented by the 
left multiplicative unitary $W\in\B(\ltwo(\G)\ot\ltwo(\G))$:
\[
\cop(x) = W^*(1\ot x)W \qquad(x\in\linfty(\G)).
\]
The reduced C*-algebra $C_0(\G)$ is the norm closure 
of 
\[
\set{(\id\ot\omega)W}{\omega\in\B(\ltwo(\G))_*}.
\]
On the other hand, the norm closure of 
\[
\set{(\omega\ot\id)W}{\omega\in\B(\ltwo(\G))_*}
\]
gives the reduced C*-algebra $C_0(\dual\G)$ 
of the dual quantum group $\dual\G$. 
The left multiplicative unitary of the dual quantum group 
is just $\hat W =\sigma W^*\sigma$ where $\sigma$ is the flip map
on $\ltwo(\G)\ot\ltwo(\G)$. 
The associated von Neumann algebras $\linfty(\G)$ 
and $\linfty(\dual\G)$ are the weak$^*$-closures of 
the respective C*-algebras  $C_0(\G)$ and $C_0(\dual\G)$.
The predual $\lone(\G)$ of $\linfty(\G)$ is a Banach algebra
under the convolution product $\omega\conv\tau = (\omega\ot\tau)\cop$.
Given $\xi,\eta\in L^2(\G)$, let $\omega_{\xi,\eta} \in L^1(\G)$ be the
normal functional $x\mapsto (x\xi\mid\eta)$.  As $L^\infty(\G)$ is in standard
position on $L^2(\G)$, every member of $L^1(\G)$ arises in this way.

The scaling group $(\tau_t)$ of $\G$ 
is implemented by a  positive, injective, densely defined 
operator $P$ on $\ltwo(\G)$: we have $\tau_t(x) = P^{it} x P^{-it}$. 
Then the antipode $S$ of $\G$ has a polar decomposition
$S = R\tau_{-i/2}$, where $R$ is the unitary antipode.
In particular $\cop R = (R\ot R)\sigma\cop$, 
where $\sigma$ denotes the flip map, now on $\linfty(\G)\vnot\linfty(\G)$.

We follow \cite[Section~3]{kustermans:universal} to define the $*$-algebra
$\lone_\sharp(\G)$.  Recall that $\omega\in \lone(\G)$ is a member of
$\lone_\sharp(\G)$ if and only if there exists $\omega^\sharp\in \lone(\G)$
such that
\[ \ip{x}{\omega^\sharp} = \overline{ \ip{S(x)^*}{\omega} }
\qquad (x\in \dom(S)) \]
where $\dom(S)$ denotes the domain of $S$.
Then $\lone_\sharp(\G)$ is a dense subalgebra of $\lone(\G)$.
The natural norm of $\lone_\sharp(\G)$ is $\|\omega\|_\sharp =
\max( \|\omega\|, \|\omega^\sharp\| )$, and with $\sharp$ as the 
involution, $\lone_\sharp(\G)$ is a Banach $*$-algebra.
For $\omega\in \lone(\G)$, let $\omega^*\in \lone(\G)$ be the functional
$\ip{x}{\omega^*} = \overline{ \ip{x^*}{\omega} }$.  Thus
$\omega\in \lone_\sharp(G)$ if and only if $w^* \circ S$ is bounded.  Notice
that as $\Delta$ is a $*$-homomorphism, the map $\omega\mapsto\omega^*$ is
an anti-linear homomorphism on $\lone(\G)$, while $\omega\mapsto\omega^\sharp$
is an anti-linear anti-homomorphism on $\lone_\sharp(\G)$.

The universal C*-algebra $C_0^u(\dual\G)$ associated to $\dual\G$ is 
the universal C*-com\-ple\-tion of $\lone_\sharp(\G)$ 
(see \cite{kustermans:universal} for details). 
The natural map (i.e.\ the universal representation) 
$\lambda_u\col \lone_\sharp(\G)\to  C_0^u(\dual\G)$ 
is implemented as $\lambda_u(\omega) = (\omega \ot\id)(\wW)$
where $\wW\in M(C_0(\G) \otimes  C_0^u(\dual\G))$ 
is the maximal unitary corepresentation of $\G$
(which is denoted by $\dual{\mc V}$ in
\cite[Proposition~4.2]{kustermans:universal}).

\begin{lemma}\label{lem:fsp_implies_pdf}
Let $x=(\id\otimes\hat\mu)(\wW^*)$ for some $\hat\mu\in C_0^u(\dual\G)^*_+$.
Then $\ip{x^*}{\omega\star\omega^\sharp}\geq 0$ for every
$\omega\in\lone_\sharp(\G)$.
\end{lemma}
\begin{proof}
Simply note that
\begin{align*} \ip{x^*}{\omega\star\omega^\sharp} &=
\ip{(\Delta\otimes\id)\wW}{\omega\otimes\omega^\sharp\otimes\hat\mu}
= \ip{\wW_{13} \wW_{23}}{\omega\otimes\omega^\sharp\otimes\hat\mu} \\
&= \ip{\hat\mu}{\lambda_u(\omega) \lambda_u(\omega^\sharp)}
= \ip{\hat\mu}{\lambda_u(\omega) \lambda_u(\omega)^*} \geq 0, \end{align*}
as required.
\end{proof}

Similarly to \cite{kraus-ruan:multipliers-of-kac},
we say that $x\in  M(C_0(\G))$ is a 
\emph{left multiplier of $\lone(\dual\G)$} if 
\[
x \hat\lambda(\hat\omega) \in \hat\lambda\bigl(\lone(\dual\G)\bigr)
\quad\text{whenever}\quad \hat\omega\in \lone(\dual\G),
\]
where
\[
\hat\lambda \col  C_0(\dual{\G})^* \to  M(C_0(\G)),
\qquad \hat\lambda(\hat\mu)
= (\hat\mu\otimes\id)(\hat\ww) = (\id \otimes \hat\mu)(\ww^*).
\]
In this case we can define $L_x\col \lone(\dual\G)\to \lone(\dual\G)$
by
\[
\hat\lambda(L_x(\hat\omega)) = x \hat\lambda(\hat\omega)
\]
because $\hat\lambda$ is injective. 
We see immediately that $L_x$ is a left multiplier (often termed a
``left centraliser'' in the literature) in the usual sense, 
that is, $L_x(\hat\omega\conv\hat\tau) = L_x(\hat\omega)\conv \hat\tau$ 
for every $\hat\omega,\hat\tau\in\lone(\dual\G)$.
The following lemma is shown for Kac algebras in
\cite{kraus-ruan:multipliers-of-kac}, but 
since we need the (short) argument once more, we include a proof.

\begin{lemma} \label{lem:L_x-bounded}
Let $x\in  M(C_0(\G))$ be a left multiplier of $\lone(\dual\G)$.
Then $L_x\col\lone(\dual \G)\to \lone(\dual \G)$ is bounded.
\end{lemma}

\begin{proof}
We apply the closed graph theorem. 
Suppose that $\hat\omega_n\to \hat\omega$ and 
$L_x(\hat\omega_n)\to \hat\tau$ in $\lone(\dual\G)$. 
Then 
\begin{align*}
\|\hat\lambda(L_x(\hat\omega)) - \hat\lambda(\hat\tau)\|
&\le \|x\hat\lambda(\hat\omega)  - x\hat\lambda(\hat\omega_n)\|
+ \|\hat\lambda(L_x(\hat\omega_n)) -  \hat\lambda(\hat\tau)\|\\
&\le
\| x\| \|\hat\omega-\hat\omega_n\| + \|L_x(\hat\omega_n) -  \hat\tau\|
\to 0
\end{align*}
as $n\to \infty$. 
Since $\hat\lambda$ is injective, we have $L_x(\hat\omega) = \hat\tau$ 
and by the closed graph theorem, $L_x$ is bounded.
\end{proof}

We say that $x\in  M(C_0(\G))$ is an \emph{$n$-positive multiplier} if
it is a left multiplier of $\lone(\G)$ and the map
$L_x^*\col \linfty(\dual\G)\to\linfty(\dual\G)$ is $n$-positive.
We shall consider $n$-positive multipliers more carefully in 
Section~\ref{sec:n-pos}, but the main interest of the paper 
shall be the completely positive multipliers:
that is, $x\in  M(C_0(\G))$ that are $n$-positive multipliers
for every $n\in\nat$. 
When $x$ is a completely positive multiplier, 
$L_x^*$ extends to a normal, completely positive map 
$\Phi\col \B(\ltwo(\G))\to \B(\ltwo(\G))$
(see \cite[Proposition~4.3]{junge-neufang-ruan:mults} or
\cite[Proposition~3.3]{daws:cp-multipliers}).
Moreover, by \cite[Proposition~6.1]{daws:cp-multipliers},
\[ 
\ip{x^*}{\omega_{\xi, \alpha} \conv \omega_{\eta, \beta}^{\sharp}}
  = \pp{ \Phi(\theta_{\xi,\eta})\beta}{ \alpha }
\]
for every $\xi,\eta\in\dom(P^{1/2})$ and
$\alpha,\beta\in\dom(P^{-1/2})$, so $x$ is completely positive
definite. We shall prove the converse in Section~\ref{sec:cpd}, but 
first we need a bit of groundwork.

The following is similar to known results about cores for analytic
generators (compare \cite[Theorem~X.49]{reed-simon:vol2} for example)
but we give the short proof for completeness.
Let us just remark that as $S=R\tau_{-i/2}$ 
and $\tau_t(x) = P^{it} x P^{-it}$ for all~$t$,
the functional $\omega_{\xi,\alpha}$ is in $L^1_\sharp(\G)$ whenever
$\xi\in D(P^{1/2})$ and $\alpha\in D(P^{-1/2})$,
and in this case $\omega_{\xi,\alpha}^\sharp
= R_*(\omega_{P^{-1/2}\alpha,P^{1/2}\xi}) = \omega_{\hat JP^{1/2}\xi,\hat J
P^{-1/2}\alpha}$; see \cite[Section~6]{daws:cp-multipliers}.

\begin{lemma}\label{lem:sharp-core}
The set
\[
D = \set{\omega_{\xi,\alpha}}{\xi\in\dom(P^{1/2}),\alpha\in\dom(P^{-1/2})}
\]
is dense in $\lone_\sharp(\G)$ with respect to its natural norm
\textup{(}i.e.\ $D$ is a core for $\omega\mapsto\omega^\sharp$\textup{)}.
\end{lemma}

\begin{proof}
For $\omega\in\lone(\G)$, $r>0$, define
\[ 
\omega(r) = \frac{r}{\sqrt\pi} 
\int_{-\infty}^\infty e^{-r^2t^2} \omega\circ\tau_t \ dt. 
\]
See also Section~\ref{sec:bai_coamen_case} below.
Since the modular group $(\tau_t)$ is implemented by $P$, 
it follows that $D$ is invariant under $(\tau_t)$.  
On the other hand, since $R$ commutes with $(\tau_t)$
and $S = R\tau_{-i/2}$, we have 
$(\omega\circ\tau_t)^\sharp = \omega^\sharp\circ\tau_t$ 
for every $\omega\in\lone_\sharp(\G)$, and so
$t\mapsto \omega\circ\tau_t$ is continuous  
with respect to the norm of $\lone_\sharp(\G)$. 
Consequently, if $\omega\in D$, 
then $\omega(r)$ is in the $\lone_\sharp(\G)$-closure 
of $D$ for every $r>0$. 

Given $\omega\in \lone_\sharp(\G)$, there exists 
$(\omega_n)\sub D$ such that $\omega_n\to \omega$ in $\lone(\G)$,
because $\linfty(\G)$ is in standard form on $\ltwo(\G)$ 
and the domains of $P^{1/2}$ and $P^{-1/2}$ are dense 
in $\ltwo(\G)$. 
By the beginning of the proof, $\omega_n(r)$ is in the
$\lone_\sharp(\G)$-closure  of $D$. A simple calculation shows that 
\[
\|\omega(r) - \omega_n(r)\|_\sharp \le e^{r^2/4}\|\omega - \omega_n\|, 
\]
and so $\omega(r)$ is in the $\lone_\sharp(\G)$-closure 
of $D$. As $r\to \infty$, we have $\omega(r)\to \omega$ 
and $\omega(r)^\sharp = \omega^\sharp(r)\to \omega^\sharp$ 
(since $\omega\in\lone_\sharp(\G)$). Therefore $\omega$ is in the
$\lone_\sharp(\G)$-closure of $D$, as claimed.
\end{proof}

\section{Density of products in $\lone_\sharp(\G)$}\label{sec:density}

The main result of this section is that, in analogy to $\lone(\G)$,
the convolution products are linearly dense in $\lone_\sharp(\G)$ with
respect to its natural norm. We shall prove this result based on two
(closely related) lemmas,  the first of which is from
\cite[Proposition~A.1]{brannan-daws-samei:cb-repn}.

\begin{lemma}\label{lem:one} 
Let $x,y\in \linfty(\G)$ satisfy $\ip{y}{\omega^\sharp} = \ip{x^*}{\omega^*}$
for all $\omega\in \lone_\sharp(\G)$.
Then $y\in \dom(S)$ and $S(y) = x^*$.
\end{lemma}

\begin{lemma}\label{lem:two}
Let $x,y\in \linfty(\G)$ satisfy $\ip{y}{(\omega_1\conv\omega_2)^\sharp}
= \ip{x^*}{\omega_1^* \conv \omega_2^*}$ for all $\omega_1,\omega_2\in
\lone_\sharp(\G)$.  Then $y\in \dom(S)$ with $S(y)=x^*$.
\end{lemma}

\begin{proof}
For $n\in\mathbb N$ define the smear 
\[ y(n) = \frac{n}{\sqrt\pi} \int_{-\infty}^\infty e^{-n^2t^2} \tau_t(y) \ dt \]
where $(\tau_t)$ is the scaling group. Define $x(n)$ similarly using $x^*$.  

As $R$ commutes with $(\tau_t)$ and $S=R\tau_{-i/2}$, it follows that
$\omega^\sharp\circ\tau_t = (\omega\circ\tau_t)^\sharp$
for $\omega\in \lone_\sharp(\G)$.
Using that $\cop\circ\tau_t = (\tau_t\otimes\tau_t)\circ\cop$
we find that for $\omega_1,\omega_2\in \lone_\sharp(\G)$
\begin{align*}
\ip{\cop(y(n))}{\omega_2^\sharp \otimes \omega_1^\sharp}
&= \frac{n}{\sqrt\pi} \int_{-\infty}^\infty e^{-n^2t^2}
\ip{\cop(\tau_t(y))}{\omega_2^\sharp \otimes \omega_1^\sharp} \ dt \\
&= \frac{n}{\sqrt\pi} \int_{-\infty}^\infty e^{-n^2t^2}
\ip{\cop(y)}{(\omega_2\circ\tau_t)^\sharp\otimes
   (\omega_1\circ\tau_t)^\sharp} \ dt \\ 
&= \frac{n}{\sqrt\pi} \int_{-\infty}^\infty e^{-n^2t^2}
\ip{\cop(x^*)}{(\omega_1\circ\tau_t)^* \otimes (\omega_2\circ\tau_t)^*} \ dt \\
&= \ip{\cop(x(n))}{\omega_1^* \otimes \omega_2^*}. \end{align*}
As $y(n)\in \dom(S)$ the von Neumann algebraic version of
\cite[Lemma~5.25]{kustermans-vaes:LCQG} shows that
\[ \ip{\cop(y(n))}{\omega_2^\sharp \otimes \omega_1^\sharp}
= \ip{\cop(S(y(n)))}{\omega_1^* \otimes \omega_2^*}. \]
Thus $\cop(S(y(n))) = \cop(x(n))$, 
and as $\cop$ is injective, $S(y(n))=x(n)$.  
Now $y(n)\rightarrow y$ in the $\sigma$-weak topology, and
$x(n)\rightarrow x^*$.  As $S$ is a $\sigma$-weakly closed operator,
it follows that $y\in \dom(S)$ with $S(y)=x^*$, as required.
\end{proof}

\begin{theorem}\label{thm:prod-dense}
Let $\G$ be a locally compact quantum group.  
Then the set $\set{\omega\conv\tau}{\omega,\tau\in \lone_\sharp(\G)}$ 
is linearly dense in $\lone_\sharp(\G)$ in its natural norm.
\end{theorem}

\begin{proof}
For a Banach space $E$, let $\overline{E}$ be the conjugate space to $E$.
For $x\in E$ let $\overline x\in \overline E$ be the image of $x$, so
$\overline{x} + \overline{y} = \overline{x+y}$ and $t\overline{x}
=\overline{ \overline{t} x}$ for $x,y\in E, t\in\mathbb C$.  We identify
$(\overline E)^*$ with $\overline{E^*}$ via $\ip{\overline\mu}{\overline x}
= \overline{\ip{\mu}{x}}$.

Then the map
\[ \lone_\sharp(\G) \to \lone(\G) \oplus_\infty \overline{\lone(\G)},
\quad \omega \mapsto (\omega,\overline{ \omega^\sharp }) \]
is a linear isometry.  Thus the adjoint 
$\linfty(\G)\oplus_1 \overline{\linfty(\G)}\to \lone_\sharp(\G)^*$ 
is a quotient map.  So any member of $\lone_\sharp(\G)^*$ is induced
by a pair $(x,\overline y)$ with $x,y\in \linfty(\G)$, and the dual pairing is
\[ \ip{(x,\overline y)}{\omega}
= \ip{x}{\omega} + \ip{\overline{y}}{\overline{\omega^\sharp}}
= \ip{x}{\omega} + \overline{ \ip{y}{\omega^\sharp} }. \]

Firstly, $(x,\overline y)=0$ if and only if
$\ip{-x^*}{\omega^*} = \ip{y}{\omega^\sharp}$
for all $\omega\in \lone_\sharp(\G)$
if and only if, by Lemma~\ref{lem:one}, $y\in \dom(S)$ with $S(y)=-x^*$.

Now let $(x,\overline y)$ annihilate all elements of the form $\omega\conv\tau$,
with $\omega,\tau\in \lone_\sharp(\G)$.  Then
\[ 0 = \ip{x}{\omega\conv\tau}
+ \overline{ \ip{y}{\tau^\sharp \conv \omega^\sharp} }
\implies
\ip{-x^*}{\omega^*\conv\tau^*} = \ip{y}{\tau^\sharp \conv \omega^\sharp}. \]
By Lemma~\ref{lem:two}, $y\in \dom(S)$ with $S(y)=-x^*$.  That is,
$(x,\overline y)=0$.  So by the Hahn--Banach theorem, the result follows.
\end{proof}

\section{Completely positive definite functions}  \label{sec:cpd}

The definition of completely positive definite functions on a locally
compact quantum group $\G$ was proposed  
by the first named author in \cite{daws:cp-multipliers}, and in this section 
we show that, as conjectured in \cite{daws:cp-multipliers},
such elements are precisely the completely positive multipliers. 
This result may be viewed as a version of Bochner's theorem 
because the completely positive multipliers are known 
by \cite{daws:cp-multipliers} to be of 
the form $(\id\otimes\hat\mu)(\wW^*)$, with
$\hat\mu\in C_0^u(\dual\G)^*_+$.

We begin with a preliminary result, also of independent interest. 

\begin{propn} \label{prop:S(x)}
Let $x\in\linfty(\G)$ be completely positive definite.
Then $x^*\in \dom(S)$ and $S(x^*) = x$.
\end{propn}

\begin{proof}
For every $\xi,\eta\in\dom(P^{1/2})$ and   $\alpha,\beta\in\dom(P^{-1/2})$
\begin{align*}
\ip{x^*}{\omega_{\xi, \alpha} \conv \omega_{\eta, \beta}^{\sharp}}
  &= \ip{ \Phi(\theta_{\xi,\eta}) }{ \omega_{\beta,\alpha} }
  = \conj{\ip{ \Phi(\theta_{\xi,\eta}^*) }{ \omega_{\beta,\alpha}^*}}
  = \conj{\ip{ \Phi(\theta_{\eta,\xi}) }{ \omega_{\alpha,\beta} }}\\
  &= \conj{\ip{x^*}{\omega_{\eta, \beta} \conv \omega_{\xi, \alpha}^{\sharp}}}
  = \ip{x}{\omega_{\eta, \beta}^* \conv \omega_{\xi, \alpha}^{\sharp *}}
  = \ip{x}{(\omega_{\xi, \alpha}\conv \omega_{\eta, \beta}^\sharp)^{\sharp *}}.
\end{align*}
It follows from Theorem~\ref{thm:prod-dense} and
Lemma~\ref{lem:sharp-core} that 
\[
\ip{x^*}{\omega} = \ip{x}{\omega^{\sharp *}}
\] 
for every $\omega\in \lone_\sharp(\G)$. 
Then it follows from Lemma~\ref{lem:one} that $x^*\in\dom(S)$ 
and $S(x^*) = x$.
\end{proof}

The following lemma shows that a GNS-type construction 
works for positive definite functions. 
Compare with \cite[Proposition~2.4.4]{dixmier:C*-algebras},
but note that $\lone_\sharp(\G)$ does not necessarily have
a bounded approximate identity, the lack of which is 
remedied by the density of products in $\lone_\sharp(\G)$
(Theorem~\ref{thm:prod-dense}).

\begin{lemma} \label{lem:GNS}
Let $x\in\linfty(\G)$ be positive definite. 
Then
\[
(\omega\mid\tau) = \ip{x^*}{\tau^\sharp \conv \omega}
\]
defines a pre-inner-product on $\lone_\sharp(\G)$;
let $\Lambda\col \lone_\sharp(\G) \to H$ be the associated 
map to the Hilbert space $H$ obtained by completion. 
Then, 
\[
\pi(\omega)\Lambda(\tau) = \Lambda(\omega\conv\tau)
\]
defines a non-degenerate $*$-representation $\pi$ of $\lone_\sharp(\G)$ on $H$. 
\end{lemma}

\begin{proof}
Such GNS-type results are usually stated for algebras with an approximate
identity (see for example \cite[Section~3.1]{dales:banach-algebras}
or \cite[Section~2.4]{dixmier:C*-algebras}) so for completeness, we give
the details in this slightly more general setting.
When $x\in \linfty(\G)$ is positive definite,
\[
\pp{\omega}{\tau} = \ip{x^*}{\tau^\sharp\conv \omega}.
\]
defines a positive sesquilinear form on $\lone_\sharp(\G)$.
As in the statement, let $H$ be the Hilbert space completion 
of $\lone_\sharp(\G)/\N$, where $\N$ denotes the associated null space,
and let $\Lambda\col \lone_\sharp(\G) \to H$ be the 
map taking $\omega\in\lone_\sharp(\G)$ to the image of $\omega + \N$ 
in $H$. We are left to show that 
\[
\pi(\omega)\Lambda(\tau) = \Lambda(\omega\conv\tau)
\]
defines a non-degenerate $*$-representation of $\lone_\sharp(\G)$ on $H$. 
If we can show that $\pi(\omega)$ is well-defined 
and bounded, then it follows easily that $\pi$ is a 
$*$-homomorphism. 

Let $r$ denote the spectral radius on $\lone_\sharp(\G)$.
By \cite[Corollary~3.1.6]{dales:banach-algebras},
\begin{align*}
\|\pi(\omega)\Lambda(\tau)\|^2 
&= \ip{x^*}{\tau^\sharp \conv \omega^\sharp \conv \omega \conv \tau}
\le r(\omega^\sharp\conv \omega)\ip{x^*}{\tau^\sharp\conv \tau}
 =  r(\omega^\sharp\conv \omega)\|\Lambda(\tau)\|^2.
\end{align*}
This shows that $\pi(\omega)$ maps $\N$ to $\N$ and hence is well-defined.
Moreover, we see that $\pi(\omega)$ defines a bounded operator on $H$. 

Finally, since $\lone_\sharp(\G)$ is 
the closed linear span of $\lone_\sharp(\G) \conv \lone_\sharp(\G)$
by Theorem~\ref{thm:prod-dense}
and $\Lambda$ is continuous, the $*$-representation $\pi$ is
non-degenerate. 
\end{proof}

\begin{theorem}
An element $x\in\linfty(\G)$ is completely positive definite 
if and only if it is a completely positive multiplier. 
In particular, every completely positive definite $x\in\linfty(\G)$
is in $ M(C_0(\G))$.
\end{theorem}

\begin{proof}
As already noted, the ``if'' part is proved in \cite{daws:cp-multipliers},
so we let $x\in\linfty(\G)$ be completely positive definite.
Let $\Phi\col\B(\ltwo(\G))\to \B(\ltwo(\G))$ 
be the associated completely positive map such that 
\[ 
\ip{x^*}{\omega_{\xi, \alpha} \conv \omega_{\eta, \beta}^{\sharp}}
    = \pp{ \Phi(\theta_{\xi,\eta})\beta}{ \alpha }
\]
whenever $\xi,\eta\in\dom(P^{1/2})$ and  $\alpha,\beta\in\dom(P^{-1/2})$.
Then $\Phi$ has a Stinespring dilation of the form
\[ 
\Phi(\theta) = V^*(\theta\otimes 1)V \qquad (\theta\in\B_0(\ltwo(\G))),
\]
where $V\col\ltwo(\G) \to \ltwo(\G)\otimes K$ is a bounded map
for some Hilbert space $K$
(see \cite[Section~5]{daws:cp-multipliers} for details).  Letting
$(e_i)$ be an orthonormal basis of $K$, we can define a family $(a_i)$
in $\B(\ltwo(\G))$ with  $\sum_i a_i^*a_i < \infty$ such that
\[ 
V\xi = \sum_i a_i\xi\otimes e_i
\]
and hence 
\[
\Phi(\theta) = \sum_i a_i^* \theta a_i. 
\]
We may also take the Stinespring dilation to be \emph{minimal} so that 
\[
\set{(\theta\ot\id)V\xi}{\xi\in \ltwo(\G), \theta\in\B_0(\ltwo(\G))}
\]
is linearly dense in $\ltwo(\G)\otimes K$.  This is equivalent to
vectors of the form
\[ 
\sum_i \ip{a_i}{\omega} \eta \otimes e_i \qquad
(\omega\in\B(\ltwo(\G))_*, \eta\in\ltwo(\G))
\]
being linearly dense; equivalently that vectors of the form
$\sum_i \ip{a_i}{\omega} e_i$ are dense in $K$ 
as $\omega\in\B(\ltwo(\G))_*$ varies.

Let $(\Lambda,\pi,H)$ be the GNS construction for $x$ from
Lemma~\ref{lem:GNS}. Then, for  
$\xi,\eta\in\dom(P^{1/2})$ and  $\alpha,\beta\in\dom(P^{-1/2})$,
\begin{align*} 
\bigpp{\Lambda(\omega_{\xi,\alpha}^\sharp)}{\Lambda(\omega_{\eta,\beta}^\sharp)}_H
 &= \ip{x^*}{\omega_{\eta,\beta} \conv \omega_{\xi,\alpha}^\sharp} 
 = \pp{ \Phi(\theta_{\eta, \xi}) \alpha}{ \beta } \\
 &= \sum_i \pp{ a_i(\alpha)}{\xi} \conj{\pp{a_i(\beta)}{\eta}}. 
\end{align*}
Let $q:\B(\ltwo(\G))_* \rightarrow L^1(\G)$ be the quotient map.
Since the functionals of the form $\omega_{\xi,\alpha}^\sharp$ 
are dense in $\lone_\sharp(\G)$ by Lemma~\ref{lem:sharp-core},  
it follows that there is an isometry
\[ v\col H \to K; \quad 
\Lambda(q(\omega)^\sharp)\mapsto \sum_i \ip{a_i}{\omega^*} e_i 
\qquad (\omega\in\B(\ltwo(\G))_*, q(\omega)\in\lone_\sharp(\G))
\]
(note that 
$\|\sum_i \ip{a_i}{\omega^*} e_i\|^2 \le 
\sum_i |\ip{a_i}{\omega^*}|^2 \le \|\omega\|^2 \|\sum_i a_i^* a_i\|$).
As we have a minimal Stinespring dilation, $v$ has dense range and is
hence unitary.
A corollary of $v$ even being well-defined is that for each $i$ and each
$\omega\in \B(\ltwo(\G))_*$, the value of $\ip{a_i}{\omega}$ depends only
on the value $q(\omega)$.
Hence $a_i\in \linfty(\G)$ for each $i$.

Consider now the non-degenerate $*$-representation of $\lone_\sharp(\G)$ on
$K$ given by $v \pi(\cdot) v^*$. By Kustermans 
\cite[Corollary~4.3]{kustermans:universal}, 
there is an associated unitary corepresentation $U$ of $\G$; 
so $U\in \linfty(\G)\vnot\B(K)$ and 
$v \pi(\omega) v^* = (\omega\otimes\id)(U)$. 
As $U$ is a unitary corepresentation,
we know that $(\omega^\sharp\otimes\id)(U) = (\omega\otimes\id)(U)^*
= (\omega^*\otimes\id)(U^*)$.  If $\omega^*\in L^1_\sharp(\G)$, it
follows that $(\omega\otimes\id)(U^*) = v \pi(\omega^{*\sharp}) v^*$, and thus 
\begin{align*}
&\Bigpp{ U^*\big(\xi \otimes \sum_i \ip{a_i}{\omega} e_i\big)}
      {\alpha \otimes e_j}
 = \Bigpp{ (\omega_{\xi,\alpha}\ot\id)(U^*)\sum_i \ip{a_i}{\omega} e_i}{e_j}\\
&\qquad = \Bigpp{ v \pi(\omega_{\alpha,\xi}^\sharp) v^*
   \sum_i \ip{a_i}{\omega} e_i}{e_j}
= \Bigpp{ v \pi(\omega_{\alpha,\xi}^\sharp) \Lambda(\omega^{*\sharp})}{e_j} \\
&\qquad = \bigpp{ v\Lambda(\omega_{\alpha,\xi}^\sharp\conv \omega^{*\sharp})}{e_j}
   = \ip{a_j}{ (\omega^* \star \omega_{\alpha,\xi})^* } 
   = \ip{a_j}{\omega\conv \omega_{\xi, \alpha}} \\
&\qquad = \Bigpp{ \sum_i (\omega\otimes\id)\cop(a_i) \xi \otimes e_i}
         { \alpha\otimes e_j}
\end{align*}
whenever $\xi\in\dom(P^{1/2})$, $\alpha\in\dom(P^{-1/2})$.
As an aside, we note that this is precisely the way in which $U^*$ is
defined  in \cite[Proposition~5.2]{daws:cp-multipliers}.
Thus, perhaps as expected, if $x$ comes from a
completely positive multiplier, then the two approaches to forming
representations agree. 

Since $U$ is unitary, we have, for every 
$\xi,\eta\in\dom(P^{1/2})$ and $\omega_1,\omega_2\in\lone_\sharp(\G)^*$,
\begin{align*} 
\pp{\xi}{\eta} \sum_i \ip{a_i}{\omega_1} \overline{\ip{a_i}{\omega_2}}
&= \Bigpp{U^*\big(\xi\otimes\sum_i\ip{a_i}{\omega_1} e_i\big)}
   {U^*\big(\eta\otimes\sum_i\ip{a_i}{\omega_2} e_i\big)} \\
&= \Bigpp{\sum_i (\omega_1\otimes\id)\cop(a_i)\xi \otimes e_i}
   {\sum_i (\omega_2\otimes\id)\cop(a_i)\eta \otimes e_i } \\
&= \sum_i \bigpp{(\omega_2\otimes\id)\cop(a_i)^* 
                 (\omega_1\otimes\id)\cop(a_i) \xi}{ \eta }.
\end{align*}
It follows that
\[ \sum_i \ip{a_i^* \otimes a_i \otimes 1}
          {\omega_2^* \otimes\omega_1\otimes \omega_{\xi,\eta}} 
= \sum_i \ip{\cop(a_i^*)_{13}  \cop(a_i)_{23}}
            {\omega_2^* \otimes\omega_1 \otimes \omega_{\xi,\eta}}. \]
As this holds for a dense collection of
$\omega_1,\omega_2, \xi,\eta$ it follows that
\[ 
\sum_i a_i^* \otimes a_i \otimes 1 = \sum_i \cop(a_i^*)_{13} \cop(a_i)_{23}
\]
(recall that $\sum_i a_i^*a_i<\infty$ so the sums on both sides do converge
$\sigma$-weakly.).
Then, for $\theta_{\xi,\eta}\in\B_0(\ltwo(\G))$ and $\alpha,\beta\in \ltwo(\G)$,
\begin{align*} 
\pp{ \Phi(\theta_{\xi,\eta})\beta }{ \alpha } 1
&= \sum_i \pp{ a_i^*\theta_{\xi,\eta} a_i\beta }{ \alpha } 1
= \sum_i \pp{ a_i^*\xi }{ \alpha } \pp{ a_i\beta }{ \eta } 1 \\
&= \sum_i (\omega_{\xi,\alpha} \otimes \omega_{\beta,\eta}\otimes\id)
         (a_i^*\otimes a_i\otimes 1) \\
&= \sum_i (\omega_{\xi,\alpha} \otimes \omega_{\beta,\eta}\otimes\id)
         \cop(a_i^*)_{13} \cop(a_i)_{23}\\
&= \sum_i (\omega_{\xi,\alpha} \otimes \id)\cop(a_i^*)
         (\omega_{\beta,\eta}\otimes\id)\cop(a_i) \\
&= \sum_i (\omega_{\beta,\alpha}\otimes\id)\big( \cop(a_i^*)
          (\theta_{\xi,\eta}\otimes 1) \cop(a_i) \big).
\end{align*}
It follows that
\[ 
\Phi(\theta)\otimes 1 = \sum_i \cop(a_i^*)(\theta\otimes 1)\cop(a_i)
  \qquad (\theta\in\B_0(\ltwo(\G))). 
\]
By normality, and using that $\Delta(\cdot) = W^*(1\otimes\cdot)W$,
we see that for $y\in\B(\ltwo(\G))$,
\[ 
\Phi(y)\otimes 1 = \sum_i W^*(1\otimes a_i^*)W(y\otimes 1)W^*(1\otimes a_i)W
\]
and hence
\begin{equation} \label{eq:Phi-L} \begin{split}
1\ot \Phi(y) &= \sum_i \hat W(a_i^*\otimes 1)\hat W^*(1\otimes y)
   \hat W(a_i\otimes 1) \hat W^* \\
&= \hat W\bigl((\Phi\ot\id)(\hat W^*(1\ot y)\hat W)\bigr)\hat W^*.
\end{split} \end{equation} 
In particular, for $\hat x\in \linfty(\dual\G)$,
\[ 
\hat W^*(1\otimes \Phi(\hat x))\hat W = (\Phi\otimes \id)\hat\cop(\hat x). 
\]
Now, the left-hand-side is a member of $\linfty(\dual\G)\vnot\B(\ltwo(\G))$,
and the right-hand-side is a member of $\B(\ltwo(\G))\vnot \linfty(\dual\G)$,
and so both sides are really in $\linfty(\dual\G) \vnot \linfty(\dual\G)$
(by taking bicommutants for example).  Then
\[ (\Phi\otimes\id)\big( \hat\cop(\hat x) (1\otimes \hat y) \big)
= \big((\Phi\otimes\id) \hat\cop(\hat x)\big) (1\otimes \hat y)
\in \linfty(\dual\G)\vnot \linfty(\dual\G), \]
and so, as 
$\set{\hat\cop(\hat x) (1\otimes \hat y)}%
{\hat x,\hat y\in \linfty(\dual\G)}$ 
is a $\sigma$-weakly, linearly dense subset of
$\linfty(\dual\G)\vnot \linfty(\dual\G)$, it follows that
$\Phi$ maps $\linfty(\dual\G)$ to $\linfty(\dual\G)$.
(We remark that the fact that $\set{\hat\cop(\hat x) (1\otimes \hat y)}%
{\hat x,\hat y\in \linfty(\dual\G)}$ is linearly $\sigma$-weakly dense
is shown directly in the remark after
\cite[Proposition~1.21]{van-daele:lcqg_vn_approach}.  One can prove this
by noting that $D(S)$ is $\sigma$-weakly dense, and then using the
von Neumann algebraic version of the characterisation of $S$ given by
\cite[Corollary~5.34]{kustermans-vaes:LCQG}.)

Let $L$ be the restriction of $\Phi$ to $\linfty(\dual\G)$.  Then the
calculation above shows that
\[ \hat\cop L = (L\otimes\id)\hat\cop, \]
and so $L$ is the adjoint of a completely positive left multiplier on
$\lone(\dual\G)$. Comparing \eqref{eq:Phi-L} with 
\cite[Proposition~3.3]{daws:cp-multipliers} we see that 
$\Phi$ coincides with the extension of $L$ used 
in \cite{daws:cp-multipliers, junge-neufang-ruan:mults}. In particular, by 
\cite[Propositions~3.2]{daws:cp-multipliers} there is 
$x_0\in  M(C_0(\G))$ such that 
\[  (x_0\otimes 1)W^* = (\id\otimes L)(W^*) = (\id\otimes \Phi)(W^*). \]
Equivalently,
\[  x_0\otimes 1 = \sum_i (1\otimes a_i^*) \cop(a_i). \]
Moreover, by \cite[Proposition~6.1]{daws:cp-multipliers}, 
$x_0$ satisfies
\[ 
\ip{x_0}{\omega_{\alpha,\eta} \conv \omega_{\xi,\beta}^{\sharp *}}
= \pp{ \Phi(\theta_{\xi,\eta})\alpha}{\beta }
= \ip{x^*}{\omega_{\xi,\beta} \conv \omega_{\eta,\alpha}^{\sharp}}
= \ip{S(x^*)}{\omega_{\alpha,\eta} \conv \omega_{\xi,\beta}^{\sharp *}}
\]
for every $\xi,\eta\in\dom(P^{1/2})$ and  $\alpha,\beta\in\dom(P^{-1/2})$
because $x^*\in\dom(S)$ and $S(x^*) = x$ by Proposition~\ref{prop:S(x)}.
It follows, by density of such $\xi,\eta,\alpha,\beta$, that 
$x_0 =  x$. 
In particular, $x$ is a completely positive multiplier.
\end{proof}

Two immediate corollaries of the proof are the following.

\begin{cor}
Let $x$ be completely positive definite, this being witnessed by the completely
positive map $\Phi$.  Then $\Phi$ restricted to $L^\infty(\dual\G)$ is the
adjoint of a completely bounded left multiplier $L$ of $L^1(\dual\G)$, and
$\Phi$ is the canonical extension of $L^*$.
\end{cor}

\begin{cor}\label{cor:sqr_sum_coeffs}
Let $x$ be positive definite, with GNS construction $(H,\Lambda,\pi)$,
and let $(f_i)$ be an orthonormal basis of $H$.
Suppose there is a dense subset $D\subseteq L^1_\sharp(\G)$ and a family
$(a_i)$ in $\mc B(L^2(\G))$ with 
$(\Lambda(\omega^\sharp)\mid f_i)_H = \ip{a_i}{\omega^*}$ 
for all $\omega\in D$ and all $i$.  If $\sum_i a_i^*a_i < \infty$,
then $x$ is a completely positive multiplier. 
\end{cor}

Supposing that $L^1_\sharp(\G)$ is separable, the Gram--Schmidt process
allows us to find an orthonormal basis of the form $f_i = \Lambda(\tau_i)$
for some sequence $(\tau_i)\subseteq L^1_\sharp(\G)$.  Then
\[ (f_i\mid\Lambda(\omega^\sharp))_H = \ip{x^*}{\omega\star\tau_i}
= \ip{(\id\ot\tau_i)\cop(x^*)}{\omega}, \]
and so we may set $a_i^* = \tau_i \star x^* := (\id\ot\tau_i)\cop(x^*)$.  
Thus the existence of $(a_i)$ is not the key fact; rather whether
$\sum_i a_i^*a_i$ converges is the key issue.  Below we shall see that
when $\G$ is coamenable, this is automatic, while
Example~\ref{ex:not-cp} shows that the condition does not always hold.

\section{Bounded approximate identity for $\lone_\sharp(\G)$}
\label{sec:bai_coamen_case}

To prove an analogue of Bochner's theorem for coamenable 
locally compact quantum groups, we shall need a 
bounded approximate identity for $\lone_\sharp(\G)$.
In this section we show that when $\G$ is coamenable,
$\lone_\sharp(\G)$ has a bounded approximate identity
in its natural norm -- in fact a contractive one. 
(The converse is obviously true.)
The proof  is inspired by Propositions~2.25 and~2.26 
in Kustermans's paper \cite{kustermans:one-parameter},
where the proof of the latter proposition is credited to 
A.~Van~Daele and J. Verding.  We prove a more general fact which is
inspired by the underlying idea in the proof (as we see it).

For $\omega\in \lone(\G)$, $z\in\complex$ and $r>0$, let
\[ \omega(r,z) = \frac{r}{\sqrt\pi} \int_{-\infty}^\infty
e^{-r^2(t-z)^2} \omega\circ\tau_t \ dt. \]
The integral converges in norm, as the function 
$t\mapsto \omega\circ \tau_t$
is norm-continuous.  A little calculation, and some complex analysis, shows
that
\[
\omega(r,z)\in \lone_\sharp(\G),\qquad
 \omega(r,z)^\sharp = \omega^*(r,z-i/2)\circ R. 
\]
See \cite[Section~5]{kustermans:LCQG-chapter} for example.  Notice that
\begin{equation}   \label{eq:sharp-norm}
\begin{split}
\| \omega(r,z)^\sharp\| &\leq \frac{r}{\sqrt\pi}
\int_{-\infty}^\infty \big| e^{-r^2(t-z+i/2)^2} \big|\, \|\omega\| \ dt
= e^{r^2(\Im(z)-1/2)^2} \|\omega\|. 
\end{split}
\end{equation}

By analogy with the definition of $L^1_\sharp(\G)$, define $M_\sharp(\G)$
to be the collection of those $\mu\in M(\G)$ such that there is $\mu^\sharp
\in M(\G)$ with
\[ \overline{ \ip{\mu}{S(a)^*} } = \ip{\mu^\sharp}{a}
\qquad (a\in D(S)\subseteq C_0(\G)). \]
As $S\circ *\circ S\circ * = \id$, it is easy to show that $\mu\mapsto
\mu^\sharp$ defines an involution on $M_\sharp(\G)$.
Set $\|\mu\|_\sharp = \max(\|\mu\|, \|\mu^\sharp\|)$ for 
$\mu\in M_\sharp(\G)$.

In the next theorem, we look at $\mu\in M(\G)$ such that $\mu\circ\tau_t=\mu$
for all $t$.  Then the constant function $F:\mathbb C\rightarrow M(\G);
z\mapsto\mu$ is holomorphic and satisfies that $F(t)=\mu\circ\tau_t$ for all
$t\in\mathbb R$.  Hence $\mu$ is invariant 
for the analytic continuation of the group $(\tau_t)$.  
In particular, $\mu\in M_\sharp(\G)$ and $\mu^\sharp = \mu^*\circ R$.

\begin{theorem}\label{thm:approx_meas_star}
Let $\mu\in M(\G)$ with $\mu\circ\tau_t=\mu$ for all $t$.
There is a net $(\omega_\alpha)$ in $L^1_\sharp(\G)$ 
such that $\|\omega_\alpha\|_\sharp \leq \|\mu\|$ for all $\alpha$,
and for every $x\in M(C_0(\G))$,
\[
\ip{\omega_\alpha}{x}\to\ip{\mu}{x}\qquad\text{and}\qquad
\ip{\omega_\alpha^\sharp}{x}\to\ip{\mu^\sharp}{x}=\overline{\ip{\mu}{R(x)^*}}.
\]
\end{theorem}
\begin{proof}
Assume without loss of generality that $\|\mu\| = 1$. 
Let $\mc F$ be the collection of finite subsets of $M(C_0(\G))$, 
and turn $\Lambda = \mc F \times \mathbb N$ into a directed set for the order
$(F_1,n_1) \leq (F_2,n_2)$ if and only if $F_1\subseteq F_2$ and $n_1\leq n_2$.
For each $\alpha=(F,n)\in\Lambda$ choose $r>0$ so that 
$1-e^{-r^2/4}< 1/n$.  Choose $N$ so that
\begin{equation} \label{eq:choice-N}
\frac{r}{\sqrt\pi} \int_{-\infty}^{-N} e^{-r^2t^2} \ dt
= \frac{r}{\sqrt\pi} \int_N^\infty e^{-r^2t^2} \ dt
\leq \frac{1}{n}.     
\end{equation}
By the Cohen Factorisation Theorem,
there are $a\in C_0(\G)$ and  $\nu\in C_0(\G)^*$ 
such that $\|a\|\le 1$, $\|\mu-\nu\|\le 1/n$ and 
$\mu = \nu a$, where $\nu a(b) = \nu(ab)$ for $b\in C_0(\G)$ 
(see, for example, \cite[Proposition~A2]{masuda-nakagami-woronowicz}).
By the Hahn--Banach and Goldstine Theorems, we can find
a net $(\sigma_i)_{i\in I}$ in $L^1(\G)$ converging weak$^*$ to $\nu$ in $M(\G)$,
and with $\|\sigma_i\| \leq \|\nu\|$ for each $i$.

As the interval $[-N,N]$ is compact and the map $t\mapsto a\tau_t(x)$ is
norm continuous for any $x\in F$ (since $t\mapsto \tau_t(x)$ is
strictly continuous), it follows that we can find $i\in I$ with
\begin{equation} \label{eq:mu-approx} 
\big| \ip{\mu}{\tau_t(x)} - \ip{\sigma_i a}{\tau_t(x)}\big| 
=\big| \ip{\nu}{a\tau_t(x)} - \ip{\sigma_i}{a\tau_t(x)} \big| 
\leq \frac{\|x\|}{n} 
\end{equation}
whenever $t\in [-N,N]$ and $x\in F$. 
Set 
\[
\omega_\alpha = \frac{e^{-r^2/4}}{1+\frac{1}{n}} (\sigma_i a)(r,0).
\]
Since $\|\sigma_i a\|\le \|\nu\| \le (1+1/n) \|\mu\|$, we have
$\|\omega_\alpha\| \le e^{-r^2/4}\|\mu\| \leq 1$ 
and also 
$\|\omega_\alpha^\sharp\| \le 1$ by~\eqref{eq:sharp-norm}.

Now given $\epsilon>0$ and $x\in M(C_0(\G))$, 
choose $\alpha_0 = (F,n_0)$ such that $\{x, R(x)^*\}\sub F$ and 
$n_0 \ge 1+7\|x\|/\epsilon$. Since $\mu\circ\tau_t = \mu$ for all $t$,
we have for every $\alpha\ge \alpha_0$ that
\begin{equation} \label{eq:big-approx}
\begin{split} \big| \ip{\mu}{x} - \ip{\omega_\alpha}{x} \big|
&= \Big| \frac{r}{\sqrt\pi} \int_{-\infty}^\infty
e^{-r^2t^2} \Big(\ip{\mu}{\tau_t(x)} - 
 \frac{e^{-r^2/4}}{1+\frac{1}{n}}\ip{\sigma_i a}{\tau_t(x)} \Big) 
  \ dt \Big|\\
&\leq \frac{4\|x\|}{n} + 
\frac{r}{\sqrt\pi} \int_{-N}^N e^{-r^2t^2} 
  \big| \ip{\mu}{\tau_t(x)} - \ip{\sigma_i a}{\tau_t(x)} \big| \ dt\\
&\qquad + \frac{r}{\sqrt\pi} \int_{-N}^N e^{-r^2t^2} 
   \bigg(1-\frac{e^{-r^2/4}}{1+\frac{1}{n}}\bigg)
     |\ip{\sigma_i a}{\tau_t(x)}| \ dt 
\end{split}
\end{equation}
by \eqref{eq:choice-N}. Now
\[
\bigg(1-\frac{e^{-r^2/4}}{1+\frac{1}{n}}\bigg)
   |\ip{\sigma_i a}{\tau_t(x)}| 
\le (1 + \frac1n - e^{-r^2/4})\|x\|
\le \frac{2\|x\|}{n}
\]
by the choice of $r$.
Continuing from \eqref{eq:big-approx} and applying also \eqref{eq:mu-approx},
we have
\begin{align*}
\big| \ip{\mu}{x} -  \ip{\omega_\alpha}{x} \big|
\leq \frac{4\|x\|}{n} + \frac{\|x\|}{n} +  \frac{2\|x\|}{n}
< \epsilon. 
\end{align*}

Now consider $\mu^\sharp = \mu^*\circ R$.  
Using the extension of $\mu$ to a strictly continuous functional 
on $ M(C_0(\G))$, we can form $\mu(r,z)$ by
an integral which converges weakly when measured against 
elements in $ M(C_0(\G))$. As $\mu\circ\tau_t=\mu$
for all $t$, it follows that $\mu(r,0)=\mu$, and so
\[ \mu^\sharp = \mu(r,0)^\sharp = \mu^*(r,-i/2)\circ R
= \frac{r}{\sqrt\pi} \int_{-\infty}^\infty e^{-r^2(t+i/2)^2} \mu^*\circ\tau_t
\circ R \ dt. \]
Thus, by a similar calculation to the above,
\begin{align*}
\big|&\ip{\mu^\sharp}{x}  - \ip{\omega_\alpha^\sharp}{x} \big| \\
&= \Big| \frac{r}{\sqrt\pi} \int_{-\infty}^\infty e^{-r^2(t + i/2)^2} 
 \Big(\conj{\ip{\mu}{\tau_t(R(x)^*)}} - 
 \frac{e^{-r^2/4}}{1+\frac{1}{n}}
   \conj{\ip{\sigma_i a}{\tau_t(R(x)^*)}} \Big) \ dt \Big|\\
&\leq \frac{4\|x\|}{n}e^{r^2/4} + 
\frac{r}{\sqrt\pi} \int_{-N}^N e^{-r^2t^2}e^{r^2/4} 
 \big|\ip{\mu}{\tau_t(R(x)^*)} 
     - \ip{\sigma_i a}{\tau_t(R(x)^*)} \big| \ dt\\
&\qquad\qquad + \frac{r}{\sqrt\pi} \int_{-N}^N e^{-r^2t^2}e^{r^2/4} 
   \bigg(1-\frac{e^{-r^2/4}}{1+\frac{1}{n}}\bigg)
     |\ip{\sigma_i a}{\tau_t(R(x)^*)}| \ dt \\
&\leq \frac{7\|x\|}{n}e^{r^2/4}
    \le \frac{7\|x\|}{n-1} <\epsilon.
\end{align*}
Hence $\omega_\alpha^\sharp \rightarrow\mu^\sharp$ weak$^*$ in $M(C_0(\G))^*$.
\end{proof}

We remark that a standard technique would allow us, in the case when
$C_0(\G)$ is separable, to replace the net $(\omega_\alpha)$ by a
sequence in the above.

\begin{theorem}\label{thm:bai-in-lone-sharp}
Let $\G$ be coamenable.  Then $L^1_\sharp(\G)$ has a contractive
approximate identity, in its natural norm.
\end{theorem}

\begin{proof}
Recall that $\G$ is coamenable if and only if there is a state
$\epsilon\in M(\G)$ with $(\id\otimes\epsilon) \Delta = 
(\epsilon\otimes\id) \Delta = \id$;
see \cite[Section~3]{bedos-tuset:amen-coamen}.  
It is easy to see that $\epsilon$ must be unique.  For each $t$, as
$\tau_t$ is a $*$-homomorphism which intertwines the coproduct,
it follows by uniqueness that $\epsilon \circ \tau_t = \epsilon$.
As above, then $\epsilon\in M_\sharp(\G)$ with $\epsilon^\sharp
= \epsilon^*\circ R = \epsilon$ (again by uniqueness).
Thus we can apply the previous theorem to find a 
net $(\omega_\alpha)$ in $L^1_\sharp(\G)$ such that 
$\|\omega_\alpha\|_\sharp\le 1$ for every $\alpha$,
and for every $x\in M(C_0(\G))$,
\[ \ip{\omega_\alpha}{x}\to\ip{\epsilon}{x}\qquad\text{and}\qquad
\ip{\omega_\alpha^\sharp}{x}\to\ip{\epsilon}{x}. \]
Now for every $\omega\in\lone(\G)$ and $y\in\linfty(\G)$
\[ \ip{\omega_\alpha \conv \omega}{y} 
= \ip{\omega_\alpha}{(\id\ot\omega)\cop(y)} 
\to \ip{\epsilon}{(\id\ot\omega)\cop(y)} \]
because $(\id\ot\omega)\cop(y)\in  M(C_0(\G))$ 
(see \cite[Lemma 5.2]{salmi:LUC} or \cite[Theorem 2.3]{runde:UCB}).
Indeed, as $W\in M\bigl(C_0(\G)\otimes \mc K(L^2(\G))\bigr)$, where 
$\mc K(L^2(\G))$ is the collection of compact operators on $L^2(\G)$,
we see that 
\[ \cop(y) = W^*(1\otimes y)W \in M\bigl(C_0(\G)\otimes \mc K(L^2(\G))\bigr)
\cap L^\infty(\G) \overline\ot L^\infty(\G), \]
thinking of both algebras concretely acting on $L^2(\G)\otimes L^2(\G)$.
In particular, if $\tau\in \mc K(L^2(\G))^*$ is any functional which,
when restricted to $L^\infty(\G)$, agrees with $\omega$, then
\[ (\id\ot\omega)\cop(y) = (\id\otimes\tau)(W^*(1\ot y)W)
\in M(C_0(\G)). \]
Moreover, arguing using the strict topology, 
for $\omega_1,\omega_2\in\lone(\G)$,
\begin{align*}
\ip{\epsilon}{(\id\ot(\omega_1\conv\omega_2))\cop(y)}
&=\ip{\epsilon}{(\id\ot\omega_1)\cop((\id\ot\omega_2)\cop(y))}\\
&= \ip{\epsilon\conv\omega_1}{(\id\ot\omega_2)\cop(y)}
= \ip{\omega_1\conv\omega_2}{y},
\end{align*}
and as convolutions are linearly dense in $\lone(\G)$
(due to $\cop$ being injective), 
we have 
\[ 
\ip{\epsilon}{(\id\ot\omega)\cop(y)} = \ip{\omega}{y}. 
\]

So $\omega_\alpha\conv\omega\to \omega$ weakly, and 
similarly $\omega\conv \omega_\alpha\to \omega$, 
$\omega_\alpha^\sharp\conv \omega\to \omega$ and 
$\omega\conv \omega_\alpha^\sharp\to \omega$, 
all in the weak topology. 
Now we apply a standard convexity argument 
to obtain a contractive approximate identity in $\lone_\sharp(\G)$.
We give a short sketch of the argument; for more details,
see for example Day \cite[pages~523--524]{day:amenable}.
First, we obtain a bounded left approximate identity $(\tau_\beta)$
consisting of convex combinations of elements in $(\omega_\alpha)$.
Since the convex combinations are taken from elements 
further and further along the net $(\omega_\alpha)$, we have
$\omega\conv \tau_\beta\to \omega$ weakly.
Thus we can iterate the construction and obtain 
a two-sided approximate identity $(\sigma_\gamma)$, 
still consisting 
of convex combinations of elements in $(\omega_\alpha)$.
Repeating this process twice more, we obtain a net $(e_\delta)$ 
such that both $(e_\delta)$ and $(e_\delta^\sharp)$ are 
a two-sided approximate identities in $\lone(\G)$. 
Then $(e_\delta)$ is a contractive approximate identity in
$\lone_\sharp(\G)$ since for example
\[
\|e_\delta\conv \omega - \omega\|_\sharp
= \max(\|e_\delta\conv \omega - \omega\|,
       \|\omega^\sharp\conv e_\delta^\sharp - \omega^\sharp\|)\to 0
\]
for every $\omega\in\lone_\sharp(\G)$.
\end{proof}

\begin{propn}
Let $\G$ be coamenable.  For any $\mu\in M_\sharp(\G)$, there is a net
$(\omega_\alpha)$ in $L^1_\sharp(\G)$ which converges weak$^*$ to $\mu$
in $M(\G)$, and with $\|\omega_\alpha\|_\sharp \leq \|\mu\|_\sharp$
for all $\alpha$.
\end{propn}
\begin{proof}
Let $(\tau_\alpha)$ be a contractive approximate identity in $L^1_\sharp(\G)$.
We know from \cite[pages~913--914]{kustermans-vaes:LCQG} that $L^1(\G)$ is a
two-sided ideal in $M(\G)$.  Thus we can set
$\omega_\alpha = \mu\conv \tau_\alpha$, and we have that $(\omega_\alpha)$ is
a net in $L^1(\G)$, bounded by $\|\mu\|$.  
As $\mu\in M_\sharp(\G)$ and 
$\tau_\alpha\in L^1_\sharp(\G)$, it is easy to see that $\omega_\alpha
\in L^1_\sharp(\G)$ with $\omega_\alpha^\sharp = \tau_\alpha^\sharp\conv \mu^\sharp$,
and so $\|\omega_\alpha^\sharp\| \leq \|\mu^\sharp\|$.

For $a\in C_0(\G)$ and $\lambda\in M(\G)$, both
$(\lambda\otimes\id)\Delta(a)$ and 
$(\id\otimes\lambda)\Delta(a)$ are members of $C_0(\G)$
(since $\lambda = \lambda'b$ for 
some $\lambda'\in M(\G)$, $b\in C_0(\G)$ and both $(b\ot 1)\cop(a)$
and $(1\ot b)\cop(a)$  are in $C_0(\G)\ot C_0(\G)$).
Hence for $a\in C_0(\G)$,
\[ \lim_\alpha \ip{\omega_\alpha}{a}
= \lim_\alpha \ip{\tau_\alpha}{(\mu\otimes\id)\Delta(a)}
= \ip{\epsilon}{(\mu\otimes\id)\Delta(a)}
= \ip{\mu}{a}, \]
as required.
\end{proof}

We do not see why the conclusions of this proposition require that $\G$ be
coamenable; however, a proof in the general case has eluded us.

\section{Bochner's theorem in the coamenable case}\label{sec:bochner}

The following is an analogue of Bochner's theorem for 
coamenable locally compact quantum groups.

\begin{theorem} \label{thm:bochner}
Suppose that $\G$ is a coamenable locally compact quantum group. 
Then $x\in \linfty(\G)$ is positive definite if and only if 
there is a positive functional $\hat\mu\in C_0^u(\dual\G)^*$ such that 
$x = (\id\ot\hat \mu)(\wW^*)$.
\end{theorem}

\begin{proof}
As we have already noted that the converse holds, we only need to show that 
a positive definite $x\in \linfty(\G)$ is of the form 
$(\id\ot\hat \mu)(\wW^*)$.

Let $(\Lambda, H, \pi)$ be as in Lemma~\ref{lem:GNS}, applied to $x$. 
By Theorem~\ref{thm:bai-in-lone-sharp},
there is a contractive approximate identity $(e_\alpha)$
in $\lone_\sharp(\G)$.  As $\|\Lambda(e_\alpha)\| \le \|x\|^{1/2}$ for
all $\alpha$, the net $\Lambda(e_\alpha)$ clusters weakly at some
$\xi\in H$.  Observe that then $\pi(\omega)\xi = \Lambda(\omega)$ 
for every $\omega\in \lone_\sharp(\G)$.

Define $\hat\mu\in C_0^u(\G)^*_+$ by 
$\hat\mu\circ\lambda_u = \omega_{\xi,\xi}\circ \pi$  
where $\lambda_u\col \lone_\sharp(\G)\to  C_0^u(\dual\G)$ 
is the universal representation defined by
$\lambda_u(\omega) = (\omega \ot\id)(\wW)$.
Then $\hat\mu$ is well-defined because $\lambda_u$ is injective
and bounded due to the universality of $\lambda_u$. 
Moreover, for every $\omega\in\lone_\sharp(\G)$,
\[
\ip{(\id\ot\hat \mu)(\wW^*)^*}{\omega} 
=\ip{\hat\mu}{\lambda_u(\omega)} = \pp{\pi(\omega)\xi}{\xi}
=\pp{\Lambda(\omega)}{\xi} = \ip{x^*}{\omega}.
\]
\end{proof}

\begin{remark}
An alternative (and much less direct) way to prove this result is to use
Corollary~\ref{cor:sqr_sum_coeffs}.  Indeed, as $\pi(\omega)\xi
= \Lambda(\omega)$, and by \cite{kustermans:universal} there is a unitary
corepresentation $U$ such that $\pi(\omega) = (\omega\otimes\id)(U)$,
it follows that $(\Lambda(\omega^\sharp)\mid f_i) = (\pi(\omega)^*\xi\mid f_i)
= (\xi\mid(\omega\otimes\id)(U)f_i) = \overline{\ip{a_i^*}{\omega}}$ with
$a_i^* = (\id\otimes\omega_{f_i,\xi})(U)$.  However, then
$\sum_i a_i^* a_i = \sum_i (\id\otimes\omega_{f_i,\xi})(U)
(\id\otimes\omega_{\xi,f_i})(U^*) = (\id\otimes\omega_{\xi,\xi})(UU^*)
= \|\xi\|^2<\infty$, as required to invoke Corollary~\ref{cor:sqr_sum_coeffs}.
\end{remark}

The following example shows that without coamenability,
positive definiteness is not enough for a Bochner type 
representation. 

\begin{example}  \label{ex:not-cp}
Let $\freegrp_2$ be the free group on two generators
and let $\G$ be the dual of $\freegrp_2$ 
(so that $\linfty(\G) = \vn(\freegrp_2)$, the group von Neumann
algebra, and $\lone(\G) = \lone_\sharp(\G) = \A(\freegrp_2)$, the
Fourier algebra).  There exists an infinite subset $E\subseteq\freegrp_2$
such that the map given by restriction of functions induces a surjection 
$\theta\col \A(\freegrp_2)\to \ell^2(E)$, see
\cite[equation~(2.1)]{leinert-german-paper}.  The adjoint 
$\theta^*\col\ell^2(E)\to \vn(\freegrp_2)$ is hence an isomorphism
onto its range.  Let $x'\in \ell^2(E)$ and set $x=\theta^*(x')$.  As
the involution on $\A(\freegrp_2)$ 
is pointwise conjugation of functions, it follows that $x$ is positive
definite if and only if $x'$ is pointwise non-negative.

For $a,b\in A(\mathbb \freegrp_2)$, we have that
\[ (\Lambda(a)\mid\Lambda(b))_H = \ip{x}{b^\sharp \star a}
= \ip{x'}{\theta(b^\sharp) \theta(a)}
= \sum_{s\in E} x'(s) a(s) \overline{b(s)}. \]
It follows that we have an isomorphism
\[ H\rightarrow \ell^2(E); \quad \Lambda(a) \mapsto
\big( a(s) x'(s)^{1/2} \big)_{s\in E}. \]
The action $\pi$ of $A(\mathbb \freegrp_2)$ on $H\cong\ell^2(E)$ is
just multiplication of functions.  It follows that the set of
coefficient functionals of $\pi$ is precisely $\ell^1(E)$, and we
can hence recover $x$ if and only if
$x'\in\ell^1(E)\subseteq\ell^2(E)$.  This immediately tells us
that the method of the proof of  Theorem~\ref{thm:bochner} fails for
any $x'\in \ell^2(E)\sm\ell^1(E)$.

Furthermore, we claim that there is no (positive) functional in 
$C_0^u(\dual\G)^* = \ell^1(\freegrp_2)$ which gives a Bochner type
representation for such $x$ (arising from $x'\in \ell^2(E)\sm\ell^1(E)$).
Indeed, suppose towards a contradiction that $\dual\mu\in\ell^1(\freegrp_2)$
is such that $x = (\id\otimes\hat\mu)(\wW^*)$ (i.e. that $x$ is the
``Fourier--Stieltjes transform of a positive measure'' $\dual\mu$).
In this setting $C_0^u(\dual\G) = C_0(\dual\G) = c_0(\freegrp_2)$ and
$\wW^* = W^* \in \mc B(\ell^2(\freegrp_2) \otimes \ell^2(\freegrp_2))$ is the
operator
\[ W^*(e_s \otimes e_t) = e_{ts} \otimes e_t, \]
where $(e_t)_{t\in\freegrp_2}$ is the canonical orthonormal basis of
$\ell^2(\freegrp_2)$.  Let $a = \omega_{\xi,\eta} \in A(\freegrp_2)$
where $\xi=(\xi_s),\eta=(\eta_t)$ are in $\ell^2(\freegrp_2)$.  Then
\begin{align*} \ip{x}{a} &= \bigpp{x(\xi)}{\eta} =
\sum_t \dual\mu_t \bigpp{W^*(\xi\otimes e_t)}{\eta\otimes e_t} \\
&= \sum_{t,s,r} \dual\mu_t \xi_s \overline{\eta_r}
\bigpp{W^*(e_s\otimes e_t)}{e_r\otimes e_t}
= \sum_{t,s} \dual\mu_t \xi_s \overline{\eta_{ts}}
= \sum_{t,r} \dual\mu_t \xi_{t^{-1}r} \overline{\eta_r}.
\end{align*}
On the other hand, 
if we view $A(\freegrp_2)$ as a space of functions on $\freegrp_2$ via the
embedding $A(\freegrp_2) \rightarrow c_0(\freegrp_2); \omega_{\xi,\eta}
\mapsto (\omega_{\xi,\eta}\otimes\id)(W)$ then that $x=\theta^*(x')$
means that 
\begin{align*} \ip{x}{a} &= \sum_{t\in E} x'(t)
\ip{ (\omega_{\xi,\eta}\otimes\id)(W) }{\omega_{e_t,e_t}} \\
&= \sum_{t\in E} \sum_{s\in\freegrp_2} x'(s) \xi_s \overline{\eta_{t^{-1}s}}
= \sum_{t\in E} \sum_{r\in\freegrp_2} x'(s) \xi_{tr} \overline{\eta_r}.
\end{align*}
If this is true for all $\xi,\eta$, it follows that
$\hat\mu_t = x'(t^{-1})$ for all $t^{-1}\in E$ and $\hat\mu_t=0$
otherwise, which is a contradiction, as $x'\in\ell^2(E)\setminus\ell^1(E)$.
(We remark that \cite{eymard} uses the embedding $A(\freegrp_2) \rightarrow
c_0(\freegrp_2); \omega_{\xi,\eta} \mapsto (\omega_{\xi,\eta}\otimes\id)(W^*)$
but this also leads to exactly the same contradiction.)
\end{example}

\section{$n$-positive multipliers} \label{sec:n-pos}

The main result of this section is a characterisation 
of  $n$-positive multipliers on coamenable locally compact quantum
groups. The proof relies on Bochner's theorem from the previous
section. 
The present section is inspired by the work of  De Canni\`ere and Haagerup 
\cite{de-canniere-haagerup:multipliers},
and the results are extensions of theirs concerning 
the commutative case when $\G = G$ is a locally compact group.

Recall that 
\[
\hat\lambda \col  C_0(\dual{\G})^* \to  M(C_0(\G)),
\qquad \hat\lambda(\dual\mu) = (\hat\mu\otimes\id)(\hat\ww)
= (\id \otimes \hat\mu)(\ww^*). \]
The one-sided universal analogue of this is 
\[ \hat\lambda^u \col  C_0^u(\dual{\G})^* \to  M(C_0(\G)),
\qquad \hat\lambda^u(\hat\mu) = (\id \otimes \hat\mu)(\wW^*).
\]

\begin{lemma} \label{lem:mult-measures}
Let $x\in  M(C_0(\G))$ be such that $x \hat\lambda(\hat\omega) \in
\hat\lambda^u\bigl(C_0^u(\dual\G)^*\bigr)$
for every $\hat\omega \in\lone(\dual\G)$. Then $x$ is a 
left multiplier of $\lone(\dual\G)$.
\end{lemma}

\begin{proof}
Let $\hat\pi\col C_0^u(\dual \G)\to  C_0(\dual \G)$ denote the canonical
quotient map. Note that 
$\hat\lambda(\hat\omega) = \hat\lambda^u(\hat\omega\circ\hat\pi)$ 
for every $\hat\omega\in \lone(\dual\G)$.  
Define $\widetilde L_x\col \lone(\dual\G) \to  C_0^u(\dual\G)^*$ by 
\[
\hat\lambda^u(\widetilde L_x (\hat\omega)) =  x \hat\lambda(\hat\omega).
\]
Since $\hat\lambda^u$ is injective (because left slices of
$\wW^*$ are dense in $C_0^u(\dual{\G})$) $\widetilde L_x$ is 
well-defined.  The proof of Lemma~\ref{lem:L_x-bounded} also works
in this universal setting, and so $\widetilde L_x$ is bounded. 

By hypothesis, for $\hat\omega,\hat\tau\in \lone(\dual\G)$, 
\[
x\hat\lambda(\hat\omega \conv \hat\tau)
= \bigl(x\hat\lambda(\hat\omega)\bigr)\hat\lambda(\hat\tau)
\in \hat\lambda\bigl(\lone(\dual\G)\bigr)
\]
because $\lone(\dual\G)\circ\hat\pi$ is an ideal in $C_0^u(\dual\G)^*$,
see for example \cite[Proposition~8.3]{daws:mults}.
Since convolution products are linearly dense in $\lone(\dual\G)$, 
the above shows that for every $\hat\omega\in\lone(\dual\G)$ 
there is a sequence $(\hat\omega_n)\sub\lone(\dual\G)$ such that 
$\hat\omega_n\to \hat\omega$ in norm and 
$x\hat\lambda(\hat\omega_n) \in \hat\lambda\bigl(\lone(\dual\G)\bigr)$.
Since $\widetilde L_x$ is bounded, 
$\widetilde L_x(\hat\omega_n) \to \widetilde L_x(\hat\omega)$ 
in $C_0^u(\dual\G)^*$.  As $\dual\pi$ is a metric surjection,
$\lone(\dual\G)\circ\hat\pi$ is isometrically isomorphic to $\lone(\dual\G)$,
and so $\widetilde L_x(\hat\omega) \in \lone(\dual\G)\circ\hat\pi$.
Hence  $x \hat\lambda(\hat\omega)\in \hat\lambda\bigl(\lone(\dual\G)\bigr)$
as claimed.
\end{proof}

The following result extends Proposition 4.3
of \cite{de-canniere-haagerup:multipliers}.

\begin{propn} \label{propn:n-pos}
Suppose that $\G$ is coamenable.
The following are equivalent for $x\in  M(C_0(\G))$
\begin{enumerate}[(1)]
\item $x$ is an $n$-positive multiplier
\item \label{item:n-pos}
for every $(\alpha_i)_{i=1}^n\in\ltwo(\G)^n$ and 
$(\omega_i)_{i=1}^n\in\lone_\sharp(\G)^n$
\[
\Bigip{x^*}{\sum_{i,j=0}^n (\omega_j\conv\omega_i^\sharp) \cdot
                        \hat\lambda(\omega_{\alpha_j,\alpha_i})^*}\ge 0.
\]
Here $\cdot$ denotes the action of $L^\infty(\G)$ on $L^1(\G)$.
\end{enumerate}
\end{propn}

\begin{proof}
Suppose that $x$ is an $n$-positive multiplier. 
Let $(\omega_i)_{i=1}^n\in\lone_\sharp(\G)^n$, and note that
$[\lambda(\omega_i\conv \omega_j^\sharp)]\ge 0$ in $
M_n(C_0(\dual\G)) \subseteq \B(\ltwo(\G)^n)$. 
Since $L_x^*$ is $n$-positive, we have for every
$(\alpha_i)_{i=1}^n\in\ltwo(\G)^n$ 
\begin{align*}
0&\le \sum_{i,j=1}^n 
\pp{L_x^*((\lambda(\omega_i\conv\omega_j^{\sharp})))\alpha_j}{\alpha_i}
=\sum_{i,j=1}^n
\ip{\lambda(\omega_i\conv\omega_j^{\sharp})}{L_x(\omega_{\alpha_j,\alpha_i})}\\
&=\sum_{i,j=1}^n\bigl((\omega_i\conv\omega_j^{\sharp})\ot
                      L_x(\omega_{\alpha_j,\alpha_i})\bigr)(W) 
=\sum_{i,j=1}^n\bigl((\omega_i\conv\omega_j^{\sharp})^{\sharp *}\ot
                      L_x(\omega_{\alpha_j,\alpha_i})\bigr)(W^*) \\
&=\sum_{i,j=1}^n\ip{\hat\lambda\bigl(L_x(\omega_{\alpha_j,\alpha_i})\bigr)}%
               {(\omega_j\conv\omega_i^{\sharp})^{*}}
= \sum_{i,j=1}^n \ip{ x \hat\lambda\bigl(\omega_{\alpha_j,\alpha_i}\bigr)}%
               {(\omega_j\conv\omega_i^{\sharp})^{*}}
\\
&= \conj{\Bigip{x^*}{\sum_{i,j=0}^n 
    (\omega_j\conv\omega_i^\sharp) \cdot
   \hat\lambda(\omega_{\alpha_j,\alpha_i})^*}},
\end{align*}
where we used the fact that 
$S\bigl((\id\ot\tau)(W)) = (\id\ot\tau)(W^*)$ for
$\tau\in\B(\ltwo(\G))_*$. 
The calculation shows that (\ref{item:n-pos}) holds.

Conversely, suppose that (\ref{item:n-pos}) holds.
For every $x \in  M_n(\linfty(\dual\G))_+$, there is 
a net $(a_\alpha)\in  M_n(C_0(\dual\G))_+$ converging weak* to $x$,
and by \cite[Lemma~3.13]{paulsen:cb-maps-and-algebras},
every $a_\alpha$ is a sum of $n$ matrices of the form 
$[b_i b_j^*]_{i,j=1}^n$ with $(b_i)_{i=1}^n\in  C_0(\dual\G)^n$.  
Hence the density of $\lone_\sharp(\G)$ in $\lone(\G)$
implies that the linear span of matrices of the form 
$[\lambda(\omega_i\conv \omega_j^\sharp)]$ with
$(\omega_i)_{i=1}^n\in\lone_\sharp(\G)^n$ is weak*-dense in 
$ M_n(\linfty(\dual\G))_+$. Therefore the calculation in the first
part of the proof  shows that $L_x$ is $n$-positive, assuming 
that $x$ is a left multiplier of $\lone(\dual\G)$. 
We shall show that $x$ is indeed a multiplier by applying
Lemma~\ref{lem:mult-measures}.
For $\alpha\in\ltwo(\G)$, write $\omega_{\alpha} = \omega_{\alpha,\alpha}$. 
Since $x$ is $1$-positive, each  
$x \hat\lambda(\omega_{\alpha})$ is positive definite.
Hence, by Theorem~\ref{thm:bochner}, 
$x \hat\lambda(\omega_{\alpha}) \in \hat\lambda_u(C_0^u(\dual\G)^*)$.
But since $\lone(\dual\G)$ is spanned by elements of the 
form $\omega_\alpha$, an application of Lemma~\ref{lem:mult-measures}
implies that $x$ is a left multiplier of $\lone(\dual\G)$. 
\end{proof}

The following result extends 
Corollary~4.4 of \cite{de-canniere-haagerup:multipliers}.

\begin{propn}\label{prop:npos_implies_pd}
Suppose that $\dual\G$ is coamenable. Every $n$-positive 
multiplier $x\in M(C_0(\G))$ is positive definite. 
\end{propn}
\begin{proof}
Since $x$ is, in particular, a $1$-positive multiplier, 
\[
\Bigip{x^*}{(\omega\conv\omega^\sharp)\cdot\hat\lambda(\omega_{\alpha})^*}\ge 0 
\]
for every $\alpha\in\ltwo(\G)$ and $\omega\in\lone_\sharp(\G)$
(by Proposition~\ref{propn:n-pos}; this part 
does not rely on coamenability).
That is, $x\hat\lambda(\omega_{\alpha})$ is positive definite.
By coamenability of $\dual\G$, there exists a 
net $(\alpha_i)$ of unit vectors in $\ltwo(\G)$ 
such that 
\[
\|W^*(\xi\ot \alpha_i) - (\xi\ot \alpha_i)\|\to 0
\]
for every $\xi\in\ltwo(\G)$ 
(\cite[Theorem~3.1]{bedos-tuset:amen-coamen}, recall that 
$\dual W = \sigma W^*\sigma$). 
For every $\xi, \eta\in \ltwo(\G)$, we have
\[
|\ip{\hat\lambda(\omega_{\alpha_i})}{\omega_{\xi,\eta}} - \ip{1}{\omega_{\xi,\eta}}|
= |\pp{W^*(\xi\ot{\alpha_i})- (\xi\ot{\alpha_i})}{\eta\ot{\alpha_i}}|,
\]
and so $\hat\lambda(\omega_{\alpha_i})\to 1$ weak* in $\linfty(\G)$. 
Since $x\hat\lambda(\omega_{\alpha_i})$ is positive definite and 
positive definiteness is preserved by weak* limits,
it follows that $x$ is positive definite.
\end{proof}

\begin{cor}
Suppose that both $\G$ and $\dual\G$ are coamenable. 
Then the following are equivalent for $x\in M(C_0(\G))$.
\begin{enumerate}[(1)]
\item $x$ is positive definite.
\item $x$ is an $n$-positive multiplier for some natural number $n$.
\item $x$ is completely positive definite.
\end{enumerate}
\end{cor}

Even without the assumptions on coamenability, 
a completely positive definite function is always 
positive definite and an $n$-positive multiplier. 
However, we have the following counterexamples:
\begin{itemize}
\item 
 Example~\ref{ex:not-cp} shows that, in general, a positive define
 function need not be an $n$-positive multiplier for every $n$. 
 In this example $\G$ is not coamenable but $\dual\G$ is. 
\item 
  De Canni\`ere and Haagerup showed 
  in \cite[Corollary~4.8]{de-canniere-haagerup:multipliers},
  that if $\G$ is the free group $\mathbb F_N$ on $N\geq 2$
  generators and $n\geq 1$, then there exist $n$-positive multipliers
  of $L^1(\dual\G)$ which are not positive definite.  
  In this example $\G$ is coamenable but $\dual\G$ is not. 
\end{itemize}

\subsection*{Acknowledgments}
We thank the anonymous referee for helpful comments which in particular
helped clarify Example~\ref{ex:not-cp}.
The first named author was partly supported by the EPSRC grant EP/I026819/1.
The second named author was partly supported by 
the Emil Aaltonen Foundation.


\begin{thebibliography}{10}

\bibitem{bedos-tuset:amen-coamen}
E.~B{\'e}dos, L.~Tuset, \emph{Amenability and co-amenability for locally
  compact quantum groups}, Internat. J. Math. \textbf{14} (2003), 865--884.

\bibitem{blecher-smith}
D.\,P. Blecher, R.\,R. Smith,
\emph{The dual of the {H}aagerup tensor product},
J. London Math. Soc. \textbf{45} (1992) 126--144.

\bibitem{brannan:approx}
M. Brannan, 
\emph{Approximation properties for free orthogonal and free unitary
  quantum groups}, J. Reine Angew. Math.
  \textbf{672} (2012) 223--251.

\bibitem{brannan:quan_auts}
M. Brannan, \emph{Reduced operator algebras of trace-preserving
  quantum automorphism groups}, preprint, arXiv:1202.5020.

\bibitem{brannan-daws-samei:cb-repn}
M.~Brannan, M.~Daws, E.~Samei, \emph{Completely bounded representations of
  convolution algebras of locally compact quantum groups}, preprint,
  arXiv:1107.2094. 

\bibitem{dales:banach-algebras}
H.~G. Dales, \emph{Banach algebras and automatic continuity}, The Clarendon
  Press, Oxford University Press, New York, 2000.

\bibitem{daws:cp-multipliers}
M.~Daws, \emph{Completely positive multipliers of quantum groups},
   to appear \emph{Internat. J. Math.}, see arXiv:1107.5244.

\bibitem{daws:mults}
M.~Daws, \emph{Multipliers, self-induced and dual Banach algebras},
Dissertationes Math. (Rozprawy Mat.) \textbf{470} (2010) 62 pp. 

\bibitem{day:amenable}
M.~M. Day, \emph{Amenable semigroups}, Illinois J. Math. 1 (1957)
  509--544.

\bibitem{de-canniere-haagerup:multipliers}
J.~De~Canni{\`e}re, U.~Haagerup, \emph{Multipliers of the {F}ourier algebras
  of some simple {L}ie groups and their discrete subgroups}, Amer. J. Math.
  \textbf{107} (1985), 455--500.

\bibitem{dixmier:C*-algebras}
J.~Dixmier, \emph{{$C\sp*$}-algebras}, North-Holland Publishing Co., Amsterdam,
  1977, Translated from the French by Francis Jellett, North-Holland
  Mathematical Library, Vol. 15.

\bibitem{effros-ruan}
E.\,G. Effros, Z.-J. Ruan,
\emph{Operator space tensor products and {H}opf convolution algebras},
J. Operator Theory \textbf{50} (2003) 131--156.

\bibitem{enock-schwartz:kac}
M.~Enock, J.-M. Schwartz, \emph{Kac algebras and duality of locally compact
  groups}, Springer-Verlag, Berlin, 1992.

\bibitem{eymard}
P.~Eymard, \emph{L'alg\`ebre de {F}ourier d'un groupe localement compact},
   Bull. Soc. Math. France \textbf{92} (1964) 181--236.

\bibitem{hu-neufang-ruan:Multsetc}
Z.~Hu, M.~Neufang, Z.-J.~Ruan,
   \emph{Multipliers on a new class of {B}anach algebras, locally
              compact quantum groups, and topological centres},
  Proc. Lond. Math. Soc. \textbf{100} (2010) 429--458.

\bibitem{junge-neufang-ruan:mults}
M.~Junge, M.~Neufang, Z.-J.~Ruan,
  \emph{A representation theorem for locally compact quantum groups},
  Internat. J. Math. \textbf{20} (2009) 377--400.

\bibitem{kalantar-neufang-ruan}
M. Kalantar, M. Neufang, Z.-J. Ruan,
\emph{Poisson boundaries over locally compact quantum groups},
preprint, arXiv:1111.5828.

\bibitem{kraus-ruan:multipliers-of-kac}
J. Kraus, Z.-J. Ruan, \emph{Multipliers of Kac algebras},
Internat. J. Math. \textbf{8} (1997) 213--248. 

\bibitem{kustermans:one-parameter}
J.~Kustermans, \emph{One-parameter representations on {C}*-algebras}, 
 preprint, arXiv:funct-an/9707009.

\bibitem{kustermans:universal}
J.~Kustermans, \emph{Locally compact quantum groups in the universal setting},
  Internat. J. Math. \textbf{12} (2001) 289--338.

\bibitem{kustermans:LCQG-chapter}
J.~Kustermans, \emph{Locally compact quantum groups}, Quantum
independent increment processes. {I}, Lecture Notes in Math.,
vol. 1865, Springer, Berlin, 2005, pp.~99--180.

\bibitem{kustermans-vaes:LCQG_VN}
J.~Kustermans, S.~Vaes, \emph{Locally compact quantum groups in the von
   {N}eumann algebraic setting},
   Math. Scand. \textbf{92} (2003) 68--92.

\bibitem{kustermans-vaes:LCQG}
J.~Kustermans, S.~Vaes, \emph{Locally compact quantum groups}, Ann. Sci.
  {\'E}cole Norm. Sup. (4) \textbf{33} (2000) 837--934.

\bibitem{kyed:prop-T}
D. Kyed, \emph{A cohomological description of property (T) for quantum groups},
 J. Funct. Anal. \textbf{261} (2011) 1469--1493. 

\bibitem{kyed-soltan}
D. Kyed, P. So\l tan, 
   \emph{Property (T) and exotic quantum group norms},
   J. Noncommut. Geom. \textbf{6} (2012) 773--800.

\bibitem{leinert-german-paper}
M.~Leinert, \emph{Faltungsoperatoren auf gewissen diskreten {G}ruppen},
  Studia Math. \textbf{52} (1974) 149--158.

\bibitem{masuda-nakagami-woronowicz}
  T. Masuda, Y. Nakagami, S.\,L. Woronowicz,
  \emph{A {$C\sp \ast$}-algebraic framework for quantum groups},
  Internat. J. Math. \textbf{14} (2003) 903--1001.

\bibitem{paulsen:cb-maps-and-algebras}
V.~Paulsen, \emph{Completely bounded maps and operator algebras},
Cambridge University Press, Cambridge, 2002.

\bibitem{reed-simon:vol2}
M.~Reed, B.~Simon,
\emph{Methods of modern mathematical physics. II. Fourier analysis,
  self-adjointness}, Academic Press, New York-London, 1975.

\bibitem{runde:UCB}
V.~Runde, \emph{Uniform continuity over locally compact quantum groups},
  J. Lond. Math. Soc. \textbf{80} (2009) 55--71.

\bibitem{salmi:LUC}
P.~Salmi, \emph{Quantum semigroup compactifications and uniform continuity on
  locally compact quantum groups}, Illinois J. Math.
\textbf{54} (2010) 469--483.


\bibitem{van-daele:lcqg_vn_approach}
A.~Van~Daele, \emph{Locally compact quantum groups. A von Neumann
  algebra approach}, preprint, arXiv:math/0602212.

\end{thebibliography}
\end{document}